\documentclass[12pt]{amsart}

\usepackage{amscd,amssymb,amsthm,verbatim}
\usepackage{enumerate}
\usepackage{bm}
\usepackage{mathrsfs, graphicx} 
\usepackage{stmaryrd}

\usepackage[%
bookmarks=true,
colorlinks,
linkcolor=blue,
urlcolor=blue,
citecolor=blue,
plainpages=false,
pdfpagelabels,
final,
breaklinks=true\between
]{hyperref}
\usepackage{hyperref}
\usepackage[numbers,sort&compress]{natbib}

\numberwithin{equation}{section}

\theoremstyle{plain}
\newtheorem{thm}{Theorem}[section]
\newtheorem{prop}[thm]{Proposition}
\newtheorem{lem}[thm]{Lemma}
\newtheorem{cor}[thm]{Corollary}

\theoremstyle{definition}
\newtheorem{defn}[thm]{Definition}
\theoremstyle{remark}

\newcommand{\norm}[1]{\left\lVert#1\right\rVert}

\DeclareFontFamily{U}{MnSymbolC}{}
\DeclareSymbolFont{MnSyC}{U}{MnSymbolC}{m}{n}
\DeclareFontShape{U}{MnSymbolC}{m}{n}{
	<-6>  MnSymbolC5
	<6-7>  MnSymbolC6
	<7-8>  MnSymbolC7
	<8-9>  MnSymbolC8
	<9-10> MnSymbolC9
	<10-12> MnSymbolC10
	<12->   MnSymbolC12}{}
\DeclareMathSymbol{\intprod}{\mathbin}{MnSyC}{'270}

\begin{document}

\title[Aghil Alaee, Martin Lesourd, and Shing-Tung Yau]{A localized spacetime Penrose inequality and horizon detection with quasi-local mass}
\author{Aghil Alaee}
\address{Aghil Alaee, Department of Mathematics, Clark University, Worcester, MA 01610, USA, Center of Mathematical Sciences and Applications, Harvard University, Cambridge, MA 02138, USA}
\email{aalaeekhangha@clarku.edu,aghil.alaee@cmsa.fas.harvard.edu}
\author{Martin Lesourd}
\address{Martin Lesourd, Black Hole Initiative\\
	Harvard University, Cambridge, MA 02138, USA}
	\email{mlesourd@fas.harvard.edu}
\author{Shing-Tung Yau}
\address{Shing-Tung Yau, Department of Mathematics\\
	Harvard University, Cambridge, MA 02138, USA}
\email{yau@math.harvard.edu}

\begin{abstract}
	For an admissible class of smooth compact initial data sets with boundary, we prove a comparison theorem between the Wang/Liu-Yau quasi-local mass of the boundary and the Hawking mass of strictly minimizing hulls in the Jang graphs of the domain. Using this, we prove a quasi-local Penrose inequality that involves these quasi-local masses of the boundary and the area of an outermost marginally outer trapped surface (MOTS) in the domain or the area of minimizing minimal surface within the Jang graphs. Moreover,  we obtain sufficient conditions for the (non)existence of a MOTS within a domain, in the spirit of the folklore hoop conjecture.
\end{abstract}

\maketitle
\section{Introduction}
For asymptotically flat, complete Riemannian manifolds with nonnegative scalar curvature, Schoen-Yau proved the fundamental Positive Mass Theorem \cite{SchoenYauI}, which states that the ADM mass of the manifold is non-negative and zero only in the case of Euclidean space. By an intricate deformation argument involving the so-called Jang graph over the manifold, they generalized this result to initial data sets satisfying the dominant energy condition \cite{SchoenYauII}. An initial data set $(\Omega,g,k)$ for the Einstein field equations consists of a Riemannian  manifold $\Omega$, endowed with a metric $g$ and a symmetric 2-tensor $k$, which together satisfy the constraint equations: 
\begin{equation}
16\pi\mu=R_{g}+(\text{Tr}_gk)^2-|k|_g^2,\qquad 
8\pi J=\text{div}_g\left(k-(\text{Tr}_gk) g\right),
\end{equation}
Here, $\mu$ (scalar) and $J$ (one-form) represent, respectively, the local energy and momentum density in $\Omega$, and in this setting the dominant energy condition is the requirement that \[\mu\geq |J|_g\] 
\indent
In the special case where $k=0$, Bray \cite{Bray} and Huisken-Ilmanen \cite{HuiskenIlmanen} indepedently proved a far reaching refinement of the Positive Mass Theorem known as the Riemannian Penrose inequality. This states that the ADM mass of any asymptotically flat manifold $(\Omega,g)$ with boundary $\partial \Omega$ composed of an outward minimizing minimal surface with area $A$ should be no smaller than $\sqrt{\tfrac{A}{16\pi}}$, with equality if and only if $(\Omega,g)$ is the Schwarzschild manifold. Huisken-Ilmanen prove this when $A$ is the area of any connected component of $\partial \Omega$ using weak inverse mean curvature flow, whereas Bray's proof, which is based on a specific conformal flow of metrics, works when $A$ is taken to be the sum of the areas of the connected components of $\partial \Omega$. \\ \indent 
Combining the weak inverse mean curvature flow of Huisken-Ilmanen \cite{HuiskenIlmanen} and Miao's smoothening \cite{Miao}, Shi-Tam \cite{ShiTamHorizons} were able to prove a quasi-local mass comparison theorem for smooth simply connected domains $\Omega$ with nonnegative scalar curvature and a connected mean convex boundary with positive Gauss curvature. In that setting they proved that the Brown-York mass of $\partial\Omega$ is bounded below by the Hawking mass of the boundary of any connected minimizing hull $E$ compactly contained in $\Omega$. By taking $\partial E$ to be an outward minimizing minimal surface, this yields a quasi-local version of the Riemannian Penrose inequality for the Brown-York mass, and the relevant rigidity statement in their setting is that equality implies that $\Omega$ is a flat domain in $\mathbb{R}^3$. \\ \indent More recently, Alaee-Khuri-Yau \cite{AlaeeKhuriYau} proved a quasi-local Penrose inequality for initial data sets that involves the Liu-Yau $m_{LY}(\partial \Omega)$ and Wang-Yau $m_{WY}(\partial \Omega)$ quasi-local masses of $\Omega$. The inner boundary of their domain is a MOTS (with possibly multiple components), and their inequalities include charge and angular momentum. The proof involves a new gluing procedure between the Jang graph of $\Omega$ and its Bartnik-Shi-Tam extension \cite{ShiTam2002}, along with a deformation of the Jang graph by Khuri \cite{Khuri}, itself inspired by a Riemannian argument of Herzlich \cite{Herzlich}. In virtue of being based on Herzlich's argument, the Penrose inequalities of \cite{AlaeeKhuriYau} include a coefficient $\tfrac{\gamma}{1+\gamma}<1$ in front of $\sqrt{\frac{A}{4\pi}}$, where the constant $\gamma$, said to be of Herzlich-type, is independent of the area of the MOTS.  \\ \indent
Here, we consider the following class of initial data sets, which we call admissible, and we prove a comparison theorem for the Wang-Yau, Liu-Yau, and Hawking quasi-local masses. 
\begin{defn}
A domain $\Omega$ is admissible if:
\begin{itemize}
\item[(i)] $\Omega$ is bounded, simply connected and has smooth connected, boundary $\partial \Omega$,
\item[(ii)] $\partial \Omega$ is untrapped $H_{\partial\Omega}> |\text{Tr}_{\partial\Omega}k|$,
\item[(iii)] $(\Omega,g,k)$ is an initial data set satisfying $\mu\geq |J|$,
\item[(iv)] the $1$-form $X$ related to the Jang graph $\bar{\Omega}$ of $\Omega$ is admissible as in Definition \ref{DefX}.
\end{itemize}
\end{defn} 
For this class of initial data sets, we prove the following.
\begin{thm}\label{thm1.1}Let $\Omega$ be admissible and $\bar{E}$ be any connected minimizing hull contained in a Jang graph $\bar\Omega$ of ${\Omega}$ with $C^{1,1}$ boundary $\partial\bar{E}$. Then the following holds: 
\begin{itemize}
\item[(i)] if the Gauss curvature of $\partial\Omega$ is positive we have 
		\begin{equation*}
		m_{LY}(\partial{\Omega})\geq m_{H}(\partial\bar{E}),
		\end{equation*}
		
		\item[(ii)]if there is an admissible time function $\tau$ we have 
		\begin{equation*}\label{mWYandmHE}
		m_{WY}(\partial{\Omega})\geq m_{H}(\partial\bar{E}).
		\end{equation*}
\end{itemize}
\end{thm}
Theorem \ref{thm1.1} implies a quasi-local Penrose inequality for initial data sets \textit{without} a Herzlich-type constant. Indeed, although an outward minimizing minimal surface in the Jang graph (which is clearly a minimizing hull therein) does not in general project to a minimal surface in $\Omega$, the area of these surfaces are larger than their projection in $\Omega$, and so as a consequence we obtain a Penrose inequality for their projection in $\Omega$.
\begin{cor}\label{cor1}
Let $\Omega$ be admissible and $S$ be a closed surface in $\Omega$. Assume that either $S$ is a MOTS and its projection in the blow up Jang graph $\bar{\Omega}$ of $\Omega$ is outward minimizing, or $S$ has an outward minimizing minimal surface projection in a Jang graph $\bar{\Omega}$ of $\Omega$.
Then the following holds:
\begin{itemize}
\item[(i)] if the Gauss curvature of $\partial\Omega$ is positive, we have
\begin{equation*}
m_{LY}(\partial{\Omega})\geq\sqrt{\frac{|S|}{16\pi}},
\end{equation*}
\item[(ii)] if there is an admissible time function $\tau$, we have
\begin{equation*}
m_{WY}(\partial{\Omega})\geq\sqrt{\frac{|S|}{16\pi}}.
\end{equation*}
\end{itemize}
\end{cor} 
Theorem \ref{thm1.1} and Corollary \ref{cor1} also lead to MOTS (non)existence results. Such problems are well known, eg.\@ Yau's problem list, pg.\@ 371-372 \cite{SchoenYauL}, and Bartnik's list, pg.\@ 260 \cite{Bartnik1989}. They also lie in the spirit of Thorne's so-called hoop conjecture, which hypothesises that black holes form in a domain $\Omega$ if and only if the enclosed mass $M(\Omega)$ satisfies $M(\Omega)\geq a C(\partial \Omega)$ where $a$ is a universal constant and $C(\partial \Omega)$ is some notion of the circumference of $\partial \Omega$. The original formulation leaves the definitions of $M(\Omega)$ and $C(\partial \Omega)$ open. \\ \indent
The first such result, due to Schoen-Yau \cite{SchoenYauIII}, gives a sufficient condition for the existence of a MOTS in $\Omega$. Yau \cite{Yau} later refined the argument by weakening the condition imposed on the matter density in favour of a slightly stronger lower bound on the mean curvature of $\partial \Omega$ - an observation which later influenced the Liu-Yau \cite{LiuYau} and Wang-Yau \cite{WangYauPMT1} definitions of quasi-local mass. Schoen-Yau \cite{SchoenYauIII} show that an initial data set $(\Omega,g,k)$, where $\partial \Omega$ smooth, connected, outer untrapped, contains a MOTS if $\mu - |J|_g \geq \Lambda$ in $\Omega$ and that $\text{Rad}(\Omega)\geq \sqrt{\frac{3}{2}}\frac{\pi}{\sqrt{\Lambda}}$, where $\text{Rad}(\Omega)$ is a geometric quantity measuring the size of $\Omega$. Getting a MOTS existence result by assuming something about the quasi-local quantities (as opposed to pointwise) has been open since their work.\\ \indent
Using quasi-local mass makes the problem harder because quasi-local mass tends to be highly non-coercive with respect to the interior geometry. Already in the Riemannian setting, it is straightforward to construct domains with the same boundary data (and thus quasi-local mass) with one containing a minimal surface and the other one not, cf.\@ Section 2 of \cite{ShiTamHorizons}.  In spite of this, Shi-Tam \cite{ShiTamHorizons} were able to use quasi-local mass to give sufficient conditions for the existence of minimal surfaces within a compact domain with boundary. They define a quantity $m_{ST}(\Omega)$ - cf.\@ Definition \ref{massshitam} - and show that $m_{BY}(\partial \Omega)\geq m_{ST}(\Omega)$ if $\Omega$ is absent of minimal surfaces. They produce a minimal surface by reversing this inequality.  \\ \indent 
In trying to generalize their ideas to initial data sets, one faces various additional obstacles (over and above the non-coercivity aforementioned). One of these is genuinely Lorentzian in nature, and seems to have been first observed by Iyer-Wald \cite{IyerWald}. They constructed a foliation of a globally hyperbolic subset of the maximally extended Schwarzschild solution that gets arbitrarily close to the $r=0$ black hole singularity but which has the property that none of its leaves contain either an outer trapped surface or a MOTS. So in spite of the region being filled with outer trapped surfaces, none of these actually register on leaves of the foliation. This means that even if the initial data set isometrically embeds as a spacelike hypersurface in a spacetime region of strong gravity, it may be that no trapped surfaces (or MOTS) actually register on this hypersurface. One can thus anticipate that the initial data set must satisfy certain conditions other than simply being isometrically embeddable as a spacelike hypersurface in a spacetime satisfying some suitable energy condition.    \\ \indent 
Based on Theorem \ref{thm1.1}, we define a new quantity $m^*(\Omega;\bar{\Omega})$ - cf.\@ Definition \ref{newmass} - which yields the following MOTS existence result. 
\begin{thm}\label{thm1.4}
Let $\Omega$ be admissible. Assume that the Gauss curvature of $\partial \Omega$ is positive. If either
\[m^*(\Omega;\bar{\Omega})> m_{LY}(\Omega)\]

or \[m^*(\Omega;\bar{\Omega})\geq \frac{1}{4}\textnormal{diam}(\partial \Omega)\]
then there is a MOTS and outer trapped surfaces in $\Omega$. The same statement holds for $m_{WY}(\partial \Omega)$ provided we replace the Gauss curvature assumption by the condition that there exists an admissible $\tau$ on $\partial \Omega$.
\end{thm} 
Among the difficulties associated with initial data sets is the lack of scalar curvature nonnegativity. This is crucial for the monotonicity of the Hawking mass under the weak inverse mean curvature flow, which indeed underpins the entire argument in \cite{ShiTamHorizons}. \\ \indent 
Another issue that arises in the spacetime but not in the Riemannian context is that MOTS are not known to minimize some functional. Thus it is not \textit{a priori} clear that notions like \textit{minimizing hulls}, defined in \cite{HuiskenIlmanen} and crucial for the minimal surface detection argument in \cite{ShiTamHorizons}, will be of any use in the spacetime context. \\ \indent
The main ingredients of the proofs of our results combine ideas and results from Schoen-Yau \cite{SchoenYauI,SchoenYauII,Yau}, Huisken-Ilmanen's weak inverse mean curvature flow \cite{HuiskenIlmanen}, the Bartnik-Shi-Tam extension of \cite{ShiTam2002}, Shi-Tam's new quasi-local mass \cite{ShiTamHorizons}, and the recent smoothening procedure of Alaee-Khuri-Yau \cite{AlaeeKhuriYau}, which itself is based on that of Miao \cite{Miao}. \\ \indent
Using $m_{ST}(\Omega)$, we can also guarantee the existence of minimal seperating spheres in $\bar{\Omega}$. 
\begin{prop}\label{prop1.5}
Let $\Omega$ be admissible. If the Gauss curvature of $\partial\Omega$ is positive and there is an isoperimetric surface $\bar{V}\subset \bar{\Omega}$ with $m_{H}(\bar{V})\geq m_{LY}(\partial \Omega)$ for a Jang graph $\bar{\Omega}$ of ${\Omega}$, then there is a seperating outward minimizing minimal sphere in $\bar{\Omega}$. The same statement holds if there is an admissible time function $\tau$ and $m_{H}(\bar{V})\geq m_{WY}(\partial \Omega)$. 
\end{prop}
An \textit{isoperimetric} surface $V\subset \Omega$ is a $C^2$ surface in $(\Omega,g)$ whose area is no more than any other $C^2$ surface enclosing the same volume. A wider class is that of \textit{locally isoperimetric} surfaces. These minimize area given enclosed volume amongst local competitors, and are also sometimes referred to as \textit{stable volume preserving CMC}. \\ \indent
Using the results above, we also obtain the following. 
\begin{prop}\label{prop1.6}
Let $\Omega$ be admissible. If the Gauss curvature of $\partial\Omega$ is positive and for a Jang graph $\bar{\Omega}$ of $\Omega$ we have either $m_{ST}(\bar{\Omega})> m_{LY}(\partial \Omega)$ or $m_{ST}(\bar{\Omega})\geq \frac{1}{4}\text{diam}(\partial \Omega)$, then there is a seperating outward minimizing minimal sphere in $\bar{\Omega}$. The same statement holds for $m_{WY}(\partial \Omega)$ if there is an admissible function $\tau$ on $\partial \Omega$. 
\end{prop}
Lastly, we note that our results along with that of \cite{ShiTamHorizons} will lead to comparison theorems involving the spacetime Bartnik mass and the Liu-Yau or Wang-Yau mass, though we have not pursued this further. \\ \indent 
Finally, in the non-existence direction, we show the following. 
\begin{prop}\label{prop1.7}
	Let $\Omega$ be admissible and $\bar{\Omega}$. Suppose that the sectional curvatures of all Jang graphs $\bar{\Omega}$ of $\Omega$ are bounded above by some constant $C^2$, $C>0$. If the boundary has positive Gauss curvature and $m_{LY}(\partial\Omega)< \frac{1}{2C}$ then there is no MOTS in $\Omega$ and $\bar{\Omega}$ is diffeomorphic to a ball in $\mathbb{R}^3$. If $m_{WY}(\partial\Omega)< \frac{1}{2C}$, the same statement holds for $m_{WY}(\partial \Omega)$ provided there is an admissible $\tau$ on $\partial \Omega$.
\end{prop}
The Riemannian version of this statement for asymptotically flat Riemannian manifold with nonnegative scalar curvature is due to Corvino \cite{Corvino}, and a quasi-local version of it using Brown-York mass is due to Alaee-Cabrera Pacheco-McCormick \cite{AlaeeCabreraPachecoMcCormick}. \\ \indent

\textbf{Acknowledgements}. The authors would like to thank Professors Yuguang Shi and Luen-Fai Tam for clarifying discussions. A. Alaee acknowledges the support of an AMS Simons Travel Grant, the Gordon and Betty Moore Foundation, and the John Templeton Foundation. M. Lesourd acknowledge the support of the Gordon and Betty Moore Foundation and the John Templeton Foundation. S.-T. Yau acknowledges the support of NSF Grant DMS-1607871.

\section{Definitions and Preliminaries}
Let $(\Omega,g)$ be a compact Riemannian manifold. The Hawking mass $m_{H}(\Sigma)$ of a surface $\Sigma\subset\Omega$ is defined as follows.
\begin{equation}
m_{H}(\Sigma)=\sqrt{\frac{|\Sigma|}{16\pi}}\left(1-\frac{1}{16\pi}\int_{\Sigma}H^2dA_{\sigma}\right),
\end{equation} 
where $\sigma$ is the induced metric on $\Sigma$ and $H$ is the mean curvature of $\Sigma$ with respect to outward normal. Let $(\Omega,g)$ be a compact Riemannian manifold with boundary $\partial\Omega$ that has positive Gauss curvature and spacelike mean curvature vector $\vec{H}=H\nu-(\text{Tr}_{\Sigma}k) n$, $\nu$ and $n$ are unit spacelike and future-directed timelike normal to $\Sigma$, respectively. Then the Liu-Yau mass \cite{LiuYau} is defined as 
\begin{equation}
m_{LY}(\partial\Omega)=\frac{1}{8\pi}\int_{\partial\Omega}\left(H_0-|\vec{H}|\right)dA_{\sigma},
\end{equation} where $H_0$ is the mean curvature of isometric embedding of $\partial\Omega$ in $\mathbb{R}^3\subset\mathbb{R}^{3,1}$ and $\sigma$ is the induced metric on $\partial\Omega$. 

In \cite{WangYauPMT1}, Wang-Yau defined a quasi-local energy for a spacelike 2-surface $\Sigma$ embedded in a spacetime $N^{3,1}$. Let $\Sigma\hookrightarrow N^{3,1}$ be a spacelike 2-surface and suppose there is an isometric embedding of $\iota:\Sigma\hookrightarrow\mathbb{R}^{3,1}$ with mean curvature vector $\vec{H}_0$ and a time function $\tau=-\left<\iota(\Sigma),\vec{T}_0\right>$, where $\vec{T}_0$ is the
designated future timelike unit Killing field on $\mathbb{R}^{3,1}$. Then the 4-tuple $(\Sigma,\sigma,|\vec{H}|,\alpha_{\bar{e}_3})$ is called Wang-Yau data set, where $\sigma$ is the induced metric on $\Sigma$, $\vec{H}$ is the mean curvature vector which is spacelike, and $\alpha$ is the connection one-form of the normal bundle of $\Sigma$ and defined as
\begin{equation}
\alpha_{\bar{e}_3}(\cdot)=\langle{}^N\nabla_{(\cdot)}\bar{e}_3,\bar{e}_4\rangle.
\end{equation}Here $\{\bar{e}_3,\bar{e}_4\}$ is the unique orthonormal frame for the normal bundle of $\Sigma$ in $N^{3,1}$ such that $\bar{e}_3$ is spacelike, $\bar{e}_4$ is future-directed timelike, and
\begin{equation}\label{con1}
\langle\vec{H},\bar{e}_3\rangle >0,\quad\quad\quad \langle\vec{H},\bar{e}_4\rangle
=\frac{-\Delta\tau}{\sqrt{1+|\nabla\tau|^2}}.
\end{equation}

Assuming $\vec{H}_0$ is spacelike, then the Wang-Yau quasi-local energy is defined to be
\begin{equation}\label{wangyauenergy}
E_{WY}(\Sigma,\iota,\tau)
=\frac{1}{8\pi}\int_{\Sigma}\left(\mathfrak{H}_0-\mathfrak{H}\right)dA_\sigma,
\end{equation}
where the \textit{generalized mean curvature} is
\begin{equation}\label{h0}
\mathfrak{H}=\sqrt{1+|\nabla\tau|^2}\langle\vec{H},\bar{e}_3\rangle
-\alpha_{\bar{e}_3}(\nabla\tau),\qquad \mathfrak{H}_0=\sqrt{1+|\nabla\tau|^2}\hat{H}_0,
\end{equation}where $\hat{H}_0$ is the mean curvature of $\hat{\Sigma}\subset\mathbb{R}^3$ which is the orthogonal project of $\iota(\Sigma)$ with respect to $\vec{T}_0$.
This definition depends on observer $(\iota,\tau)$ and the Wang-Yau quasi-local energy is non-negative for \textit{admissible} observers \cite{WangYauPMT1}. An observer is \textit{admissible} or the time function $\tau$ is \textit{admissible} if the convexity condition \begin{equation}\label{convexitycondition}
\left(1+|\nabla\tau|^2\right)K_{\hat{\sigma}}=K_{\Sigma}
+\left(1+|\nabla\tau|^2\right)^{-1}\text{det}(\nabla^2\tau)>0,
\end{equation}
where $K_{\hat{\sigma}}$ is the Gauss curvature of metric $\hat{\sigma}=\sigma+d\tau^2$ on $\hat{\Sigma}$, is satisfied, $\Sigma$ arises as the untrapped boundary of a spacelike hypersurface $(\Omega,g,k)\hookrightarrow N^{3,1}$, and the generalized mean curvature is positive $\mathfrak{H}(e_{3}',e_{4}')>0$ for the normal bundle frame $\{e_{3}',e_{4}'\}$ determined by the solution of Jang's equation, see Definition 5.1 of \cite{WangYauPMT1}. In analogy with special relativity, the Wang-Yau mass is defined as the
infimum of energy over all admissible observers $(\iota,\tau)$, that is
\begin{equation}
m_{WY}(\Sigma)=\inf_{(\iota,\tau)}E_{WY}(\Sigma,\iota,\tau).
\end{equation}
We need following definition of minimizing hulls \cite{HuiskenIlmanen} in the next section. 
\begin{defn}\label{Defminhull}
	Let $E$ be a set in $\Omega$ with locally finite perimeter. $E$ is said to be a minimizing hull of $\Omega$ if $|\partial^*E\cap W|\leq|\partial^*F\cap W|$ for any set $F$ with locally finite perimeter such that $F\supset E$ and $F\,\backslash\, E\subsetneq\Omega$ for any set $W\subset\Omega$ containing $F\,\backslash\, E$.  Here $\partial^*E$ and $\partial^*F$ are the reduced boundaries of $E$ and $F$ respectively. $E$ is said to be strictly minimizing hull if equality (for all $W$) implies $E\cap W=F\cap W$ a.e.
\end{defn}
Note that an \textit{outward minimizing minimal surface} $\Sigma$ is a boundary of minimizing hull with zero mean curvature $H_{\Sigma}=0$. This surface represents black hole apparent horizon in Riemannian setting, which the second fundamental form of initial data set vanishes. However, for general initial data $(\Omega,g,k)$, a black hole apparent horizon is represented by a marginally outer trapped surface (MOTS) $\Sigma$ in $\Omega$ and defined as a hypersurface embedded in $\Omega$ with $H_{\Sigma}+\text{Tr}_\Sigma(k)=0$. \\ \indent
We now come to the definition of the Shi-Tam mass $m_{ST}(\Omega)$, originally described in \cite{ShiTamHorizons}.
\begin{defn}\label{massshitam}
Let  $\Omega_{1}\subsetneq \Omega_2\subset \Omega$ such that $\Omega_1$ and $\Omega_2$ have smooth boundaries. Let $\mathcal{F}_{\Omega_2}$ be the family of connected minimizing hulls, with $C^2$ boundary, of $\Omega_2$. Define 
\begin{equation}
m(\Omega_1;\Omega_2)=\sup_{E\in \mathcal{F}_{\Omega_2}, E\subset\Omega_1}m_{H}(E)
\end{equation}Then the Shi-Tam quasi-local mass is defined 
\begin{equation}
m_{ST}(\Omega)= \sup_{\Omega_1,\Omega_2}\alpha_{\Omega_1,\Omega_2}m(\Omega_1;\Omega_2),
\end{equation} where
\begin{equation}
\alpha_{\Omega_1;\Omega_2}^2=\min\left\{1,\frac{CK^{-2}\int_0^r\tau^{-1}\sin(K\tau)^2d\tau}{|\partial\Omega_1|}\right\},
\end{equation} and $K>0$ is an upper bound for sectional curvature of $\Omega_2$ and $C$ is an absolute positive constant.
\end{defn}Note that we remove the precompactness property of minimizing hull from above definition to include structure of blow up Jang graph. The constant $\alpha_{\Omega_1,\Omega_2}$ is related to Meeks and Yau \cite{MeeksYau} estimate for area of the minimal surface part of any strictly minimizing hull $E'$ of $E$ with respect to $\Omega$ in above definition. Shi and Tam \cite[Theorem 2.4]{ShiTamHorizons} show that $m_{ST}(\Omega)$ has various pleasant properties. If $\Omega$ has non-negative scalar curvature $R_g\geq 0$, and $\partial \Omega$ (which is smooth and connected) has positive mean curvature $H_{\partial \Omega}>0$ and Gauss curvature $K_{\partial \Omega}>0$, then $m_{ST}(\Omega)\geq 0$ and moreover equality is achieved if and only if $\Omega$ is a domain in flat $\mathbb{R}^3$. \\ \indent 
We now consider definitions for initial data sets $(\Omega,g,k)$ following the classic arguments of Schoen and Yau \cite{SchoenYauII}. Recall that in their proof of the spacetime positive mass theorem, they are faced with the issue that an initial data set satisfying the dominant energy condition $\mu\geq|J|_g$ that may not have a non-negative scalar curvature. They overcome this by considering deformations of the initial data $g\rightarrow \bar{g}=g+df^2$, where $f:\Omega\rightarrow\mathbb{R}$ is a solution of the Jang equation and $\bar{g}$ is the metric induced on the Jang graph of the initial data set. Upon proving existence of solution of the Jang's equation, they are able to conformally deform the Jang graph (without significantly affecting mass) to another initial data set with zero scalar curvature, which in turn eventually permits for an application of the Riemannian positive mass theorem \cite{SchoenYauI}. In the current setting we consider the following Jang's equation with Dirichlet boundary condition:
\begin{equation}\label{Jang}
\begin{cases}
\left(g^{ij}-\frac{f^if^{j}}{1+|\nabla f|^2_g}\right)\left(\frac{\nabla_{ij}f}{\sqrt{1+|\nabla f|^2_g}}-k_{ij}\right)=0 & \text{in $\Omega$}\\
f=\tau & \text{on $\partial\Omega$}
\end{cases},
\end{equation}
where $f^i=g^{ij}f_j$ and the covariant derivative $\nabla$ is with respect to $g$. Assume the boundary is untrapped $H_{\partial\Omega}> |\text{Tr}_{\partial\Omega}k|$. If there is no MOTS in $\Omega$, the Dirichlet problem \eqref{Jang} has a unique smooth solution by Schoen and Yau \cite{SchoenYauIII}, which we call it the Jang graph $\bar{\Omega}$ of $\Omega$. Moreover, if there is a MOTS in $\Omega$, there is no uniqueness for the Dirichlet problem \eqref{Jang} but there exist a smooth solution by Andersson and Metzger \cite{AnderssonMetzger} which blow up in the form of a cylinder over MOTS, with $f\to\infty$ ($-\infty$) at MOTS depending on whether it
is a future (or past) MOTS. 

Following \cite{LiuYau,WangYauPMT1}, in the case of the Wang-Yau mass, we set $f=\tau$ on $\partial\Omega$ which $\tau$ is admissible time function and for the Liu-Yau mass, we set $f=\tau=0$ on $\partial\Omega$ which means the metrics $\bar{g}$ and $g$ are the same on the boundary.  Moreover the equation implies that the scalar curvature of the Jang metric is weakly nonnegative, that is
\begin{equation}\label{Jangscalar}
\bar{R}=2\left(\mu-J(w)\right)+|h-k|^2_{\bar{g}}
+2|X|_{\bar{g}}^2-2\text{div}_{\bar{g}}X\geq 2|X|_{\bar{g}}-2\text{div}_{\bar{g}}X,
\end{equation}
where $h$ is second fundamental form of the graph $t=f(x)$ in the product manifold $(\Omega\times\mathbb{R},g+dt^2)$, and $w$, $X$ are 1-forms given by
\begin{equation}\label{def-h-w-X}
h_{ij}=\frac{\nabla_{ij}f}{\sqrt{1+|\nabla f|^2_g}},\qquad w_i=\frac{f_i}{\sqrt{1+|\nabla f|^2_g}},\:
\end{equation} and 
\begin{equation}
 X_i=\frac{f^j}{\sqrt{1+|\nabla f|^2_g}}\left(h_{ij}-k_{ij}\right).
\end{equation}

We denote the Jang graph by $(\bar{\Omega},\overline{g},\bar{k}\equiv h-k,X)$. Let $\left(\operatorname{div}_{\bar{{g}}}X\right)_+$ be positive part of $\operatorname{div}_{\bar{{g}}}X$, then we have the following definition for admissible one-form $X$.
\begin{defn}\label{DefX}
	Let $(\bar{\Omega},\overline{g},\bar{k},X)$	be a Jang graph of an initial data set $(\Omega, g,k)$. The one-form $X$ is admissible if $X(\bar{\nu})>0$ on $\partial\bar{\Omega}$ and there exist constants $d_i>0$ such that 
	\begin{equation}
	\norm{\left(\operatorname{div}_{\bar{{g}}}X\right)_+}_{L^1(\bar{\Omega})}<d_1,\quad \norm{\left(\operatorname{div}_{\bar{{g}}}X\right)_+}_{L^{3/2}(\bar{\Omega})}<d_2,
	\end{equation}
	\begin{equation}
	 \norm{\left(\operatorname{div}_{\bar{{g}}}X\right)_+}_{L^{6/5}(\bar{\Omega})}<d_3,
	\end{equation}and for sufficiently small $\delta>0$ they satisfy equations \eqref{in3} and \eqref{A1}.
\end{defn}
With these definitions, we can also define a different mass $m^*(\Omega,\bar{\Omega})$, where $\bar{\Omega}$ stands for a Jang graph over $\Omega$. 
\begin{defn}\label{newmass}Let $\bar{\Omega}_1\subsetneq\bar{\Omega}_2\subset\bar{\Omega}$ with smooth boundaries, $C$ be an absolute positive constant, and $K$ be an upper bound of the sectional curvature of $\bar{\Omega}_2$. Define 
\begin{equation}
m^*(\Omega;\bar{\Omega})= \sup_{\bar{\Omega}_1,\bar{\Omega}_2} \alpha^*_{\bar{\Omega}_1,\bar{\Omega}_2} m^*(\bar{\Omega}_1;\bar{\Omega}_2),
\end{equation} 
where
\begin{equation}
{\alpha^*}^2_{\bar{\Omega}_1,\bar{\Omega}_2}=\emph{min} \left\{ 1,\frac{CK^{-2}\int_0^r\tau^{-1}\sin(K\tau)^2d\tau}{{|\partial \Omega_1|}{\beta_{\Omega,\bar{\Omega}}} } \right\},
\end{equation} and 
\begin{equation}
 \beta_{\Omega,\bar{\Omega}}= \frac{\textnormal{Rad}(\bar{\Omega})}{\textnormal{Rad}(\Omega)},\qquad m^*(\bar{\Omega}_1;\bar{\Omega}_2)= \sup_{\bar{E}\in F^*_{\bar{\Omega}_2},\bar{E}\subset \bar{\Omega}_1} m_H(\partial \bar{E}),
\end{equation}
where $F^*_{\bar{\Omega}_2}$ is the family of connected minimizing hulls, with $C^2$ boundary, of $\bar{\Omega}_2$ such that for any $\bar{E}\in F^*_{\bar{\Omega}_2}$ we have $\partial \bar{E}\cap \partial \bar{V}\neq \emptyset$ for some connected strictly minimizing hull $\bar{V}$ of $\bar{\Omega}$ with connected boundary. 

\end{defn}
By $\text{Rad}(\Omega)$ we mean the definition of \cite{SchoenYauIII}. 
\begin{defn} 
Let $\Omega$ be a domain with boundary $\partial \Omega$ and $\Gamma$ be a simple closed curve in $\Omega$ that bounds a disc. Let $N_r(\Gamma)$ be set of all points within a $r$ radius of $\Gamma$. Define $\text{Rad}(\Omega,\Gamma)$ to be
\begin{equation}
\sup\{ r : \text{dist}(\partial \Omega, \Gamma)>r \: \text{and $\Gamma$ does not bound a disc in $N_r(\Gamma)$}\}
\end{equation}Then $\text{Rad}(\Omega)=\sup_\Gamma \text{Rad}(\Omega,\Gamma)$.
\end{defn} 
In particular, $\text{Rad}(\Omega)$ may be described as the radius of the
largest torus that can be embedded in $\Omega$. For example for a ball with radius $R$ in $\mathbb{R}^3$, $\text{Rad}(\Omega)=R/2$ and for a cylinder $\mathbb{S}^2_{R}\times(-L,L)$, $\text{Rad}(\Omega)=\min\{\pi R/2,L\}$. Note that this definition is constructed such that the radius of a blow up Jang graph $\bar{\Omega}$ for a bounded domain $\Omega$ stays finite. Moreover, this new quantity $m^*(\Omega;\bar{\Omega})$ is non-negative, and $m^*(\bar{\Omega}_1;\bar{\Omega}_2)\geq m(\bar{\Omega}_1;\bar{\Omega}_2)$ since the latter involves more competitors. If the Jang graph does not blow up, since $\text{Rad}(\Omega)\leq \text{Rad}(\bar{\Omega})$, the constant satisfies $\beta_{\Omega,\bar{\Omega}}\geq 1$. When blow up occurs, there is no known general relation between $\text{Rad}(\Omega)$ and $\text{Rad}(\bar{\Omega})$.

\section{Proof of Main Results}
We start with an admissible domain $\Omega$ and consider solutions to Jang's equation with prescribed boundary data $\bar{g}=\hat{\sigma}$ on $\partial \bar{\Omega}$. Solutions to this Dirichlet boundary problem exist by Schoen-Yau \cite{SchoenYauII} and each defines a graph $\bar{\Omega}$ over $\Omega$ with the properties aforementioned. By admissible condition $X(\nu)>0$ on $\partial\bar{\Omega}$, we construct a Bartnik-Shi-Tam extension \cite{AlaeeKhuriYau,ShiTam2002} to \(\partial \bar{\Omega}\), denoted $(M_+,g_+,k_+=0,X_+=0)$, which an asymptotically flat Riemannian manifold with zero scalar curvature and boundary $\partial M_+=\partial\bar{\Omega}$ has mean curvature $\bar{H}>0$ and induced metric $\hat{\sigma}=g_+|_{\partial M_+}$. By Shi-Tam \cite{ShiTam2002}, we have the following inequality for the ADM mass of extension.
\begin{equation}\label{shitamin}
\frac{1}{8\pi}\int_{M_+}\left(\hat{H}_0-\bar{H}\right)dA_{\hat{\sigma}}\geq m_{ADM}(g_+),
\end{equation}where $\hat{H}_0$ is the mean curvature of isometric embedding of $(\partial\bar{\Omega},\hat{\sigma})$ in Euclidean space $\mathbb{R}^3$. In contrast to our setting, the extension in \cite{AlaeeKhuriYau,LiuYau,WangYauPMT1} has mean curvature $\bar{H}-X(\bar{\nu})$ for boundary $\partial\bar{\Omega}$. Next, we attach the Jang graph $(\bar{\Omega},\bar{g},\bar{k},X)$ to this Bartnik-Shi-Tam extension and denote it by
\begin{equation}
(\bar{\mathbf{M}},\bar{\mathbf{g}},\bar{\mathbf{k}},\mathbf{X})=(\overline{\Omega}\cup M_+,\overline{g}\cup g_+,\bar{k}\cup k_+,X\cup X_+),
\end{equation} which will in general have corner along $\partial\bar{\Omega}$. Since the mean curvature is the same along boundary $\partial\bar{\Omega}$, we modify general gluing developed by Alaee-Khuri-Yau \cite{AlaeeKhuriYau} for the Jang graph and have the following result.
\begin{lem}\label{lem1}
	There exists a smooth deformation $(\bar{\mathbf{M}},\bar{\mathbf{g}}_{\delta},\bar{\mathbf{k}}_{\delta},
	\mathbf{X}_{\delta})$ which differs from the original data $(\bar{\mathbf{M}},\bar{\mathbf{g}},\bar{\mathbf{k}},\mathbf{X})$ only on a $\delta$-tubular neighborhood of $\partial\bar{\Omega}$, that is $\mathcal{O}_{\delta}=\left[-\frac{\delta}{2},\frac{\delta}{2}\right]\times \partial\bar{\Omega}$, and satisfies
	\begin{equation}\label{6yj}
	\left(\bar{\mathbf{R}}_{\delta}-2|\mathbf{X}_{\delta}|_{\bar{\mathbf{g}}_{\delta}}^2
	-|\bar{\mathbf{k}}_{\delta}|^2_{\bar{\mathbf{g}}_{\delta}}\right)(t,x)=O(1),\qquad (t,x)\in\mathcal{O}_{\delta},
	\end{equation}
	\begin{equation}
	\operatorname{div}_{\bar{\mathbf{g}}_{\delta}}\mathbf{X}_{\delta}(t,x)=O(\delta^{-1/3}),\qquad (t,x)\in\mathcal{O}_{\delta}
	\end{equation}
	where  $\bar{\mathbf{R}}_{\delta}$ is the scalar curvature of $\bar{\mathbf{g}}_{\delta}$, and $O(1)$ is a constant depends on the initial data set and independent of $\delta$.
\end{lem}

\begin{proof}It is shown in \cite[Section 3]{Miao}, if the mean curvature along the corner is the same, then 	
	\begin{equation}
	\bar{\mathbf{R}}_{\delta}(t,x)=O(1), \qquad\qquad \text{as $\delta\to 0$},
	\end{equation}where the deformation region for the metric is a $\delta$-tubular neighborhood $(t,x)\in\mathcal{O}_{\delta}=\left[-\frac{\delta}{2},\frac{\delta}{2}\right]\times \partial\bar{\Omega}$. Next, near the corner surface $\partial\bar{\Omega}$, the 1-form is
	\begin{equation}
	\mathbf{X}=\mathbf{X}_t dt+\mathbf{X}_i dx^i,
	\end{equation} where $t$ is the geodesic normal coordinate and $x^i$ are coordinates on $\partial\bar{\Omega}$. We denote the deformation by
	\begin{equation} 
	\mathbf{X}_{\delta}= \mathbf{X}_{\delta t} dt + \mathbf{X}_{\delta i} dx^i.
	\end{equation}The deformation $\mathbf{X}_{\delta i}$ is similar to \cite[Lemma 5.1]{AlaeeKhuriYau} and its norm and tangential derivative are bounded. However, $\mathbf{X}_{\delta t}$ is different as following. By definition of $\mathbf{X}$, we have 
	\begin{equation}
	\mathbf{X}_{t}=X(\bar{\nu}),\qquad t<0, \qquad\text{and}\qquad \mathbf{X}_{t}=0,\qquad t>0.
	\end{equation}Let $\varsigma(t)>0$ be a smooth cut-off function defined as
	\begin{equation}
	\varsigma(t)=\begin{cases}
	1 &t\leq -\frac{\delta}{2}\\
	\varsigma'(t)\leq 0,\,|\varsigma'(t)|\leq \delta^{-1/3} & -\frac{\delta}{2}<t<-\frac{\delta^2}{200}\\
	0 &t\geq-\frac{\delta^2}{200}
	\end{cases}.
	\end{equation}Define the deformation $\mathbf{X}_{\delta t}(t,x)=\varsigma(t)X(\bar{\nu})(t,x)$ for all $t\in(-\delta,\delta)$. Then the divergence of the smoothed 1-form is 
	\begin{align}\label{p}
	\begin{split}
	\text{div}_{\bar{\mathbf{g}}_{\delta}}\mathbf{X}_{\delta}(t,x)&=\bigg(\partial_t \mathbf{X}_{\delta t}+\frac{1}{2}\mathbf{X}_{\delta t}\partial_t \log\det\gamma_{\delta}
	\\
	&+\frac{1}{\sqrt{\det \gamma_{\delta}}}
	\partial_i \left(\sqrt{\det\gamma_{\delta}}\gamma_{\delta}^{ij}
	\mathbf{X}_{\delta j}\right)\bigg)(t,x)\\
	&=\varsigma(t)'X(\bar{\nu})(t,x)+O(1)\\
	&=O(\delta^{-1/3}),\qquad (t,x)\in\mathcal{O}_{\delta}.
	\end{split}
	\end{align} where $\gamma_{\delta}$ is deformed metric on $\partial\bar{\Omega}$ defined in equation (11) of \cite{Miao}. The smoothing of $\bar{\mathbf{k}}_{\delta}$  is similar to \cite[Lemma 5.1]{AlaeeKhuriYau}.
\end{proof}

This deformation readies our initial data for a conformal deformation which will further improve the scalar curvature of the deformed Jang graph, and moreover which will do so without changing the ADM mass by much.

\begin{prop}\label{prop2}
	Given the deformation $(\bar{\mathbf{M}},\bar{\mathbf{g}}_{\delta},\bar{\mathbf{k}}_{\delta}, \mathbf{X}_{\delta})$ of Lemma \ref{lem1} and let $\bar{F}\subsetneq\bar{\Omega}\subset\bar{\mathbf{M}}$. For sufficiently small $\delta>0$, there exist a $C^2$ positive function $u_{\delta}\geq 1$ such that the conformal metric $\hat{{\mathbf{g}}}_{\delta}=u^4_{\delta}\bar{\mathbf{g}}_{\delta}$ has non-negative scalar curvature and satisfies 
	\begin{equation}\label{mass2}
	m_{ADM}(\hat{\mathbf{g}}_{\delta})\leq 	m_{ADM}(g_{+})+\frac{1}{8\pi}\int_{\partial\bar{\Omega}}X(\bar{\nu})dA_{\hat{\sigma}}.
	\end{equation}
\end{prop}
\begin{proof}Let $\kappa=K_{\delta-}+2\left(\operatorname{div}_{\bar{\mathbf{g}}_{\delta}}\mathbf{X}_{\delta}\right)_+$ and $K_{\delta}=\bar{\mathbf{R}}_{\delta}-2|\mathbf{X}_{\delta}|_{\bar{\mathbf{g}}_{\delta}}^2
	+2\operatorname{div}_{\bar{\mathbf{g}}_{\delta}}\mathbf{X}_{\delta}-|\bar{\mathbf{k}}_{\delta}|^2_{\bar{\mathbf{g}}_{\delta}}$. Moreover, since $\varsigma(t)'\leq 0$ in Lemma \ref{lem1} and $X(\bar{\nu})> 0$ on $\mathcal{O}_{\delta}$, we have
	\begin{equation}\label{Kbound}
	\begin{cases}
	K_{\delta-}=0&\text{outside $\mathcal{O}_{\delta}$}\\
	|K_{\delta-}|\leq C_1\delta^{-1/3}&\text{inside $\mathcal{O}_{\delta}$}.\\
	\end{cases},
	\end{equation}
	where $C_1$ depends on initial data sets and independent of $\delta$. Combining Lemma \ref{lem1} and Definition \ref{DefX}, we have
	\begin{equation}\label{dbounds}
\norm{\left(\operatorname{div}_{\bar{\mathbf{g}}_{\delta}}X_{\delta}\right)_+}_{L^1(\bar{\Omega})}<d_1,\qquad \norm{\left(\operatorname{div}_{\bar{\mathbf{g}}_{\delta}}X_{\delta}\right)_+}_{L^{3/2}(\bar{\Omega})}<d_2,\\ 
	\end{equation}
	\begin{equation}\label{dbounds''}
	 \norm{\left(\operatorname{div}_{\bar{\mathbf{g}}_{\delta}}X_{\delta}\right)_+}_{L^{6/5}(\bar{\Omega})}<d_3,
	\end{equation}
	for sufficiently small $\delta$. Next we consider the following PDE
	\begin{equation}\label{PDE2}
	\begin{cases}
	\Delta_{\bar{\mathbf{g}}_{\delta}}w_{\delta}+\frac{1}{8}\kappa w_{\delta}=-\frac{1}{8}\kappa & \bar{\mathbf{M}}\,\backslash \bar{F}\\
	w_{\delta}=0& \infty\\
	\bar{\nu}(w_{\delta})=0& \partial \bar{F}
	\end{cases}.
	\end{equation}

	Following Lemma 3.2 and Lemma 3.3 of \cite{SchoenYauI}, it is enough to prove an $L^6(\bar{\mathbf{M}}\,\backslash\, F)$-bound for $w_{\delta}$ to show existence of a positive $C^2(\bar{\mathbf{M}}\,\backslash\, F)$ solution $w_{\delta}$ such that $w_{\delta}=\frac{A_{\delta}}{|x|}+O(|x|^{-2})$ as $|x|\to\infty$. We multiply the PDE in \eqref{PDE2} by $w_{\delta}$ and integrate by parts.
	\begin{equation}
\int_{\bar{\mathbf{M}}\,\backslash\, \bar{F}}|\nabla w_{\delta}|^2 dV_{\bar{\mathbf{g}}_{\delta}}=\frac{1}{8}\int_{\bar{\mathbf{M}}\,\backslash\, \bar{F}}\kappa w_{\delta}^2 dV_{\bar{\mathbf{g}}_{\delta}}-\frac{1}{8}\int_{\bar{\mathbf{M}}\,\backslash\, \bar{F}}\kappa w_{\delta} dV_{\bar{\mathbf{g}}_{\delta}}.
	\end{equation}Then using \eqref{Kbound}, \eqref{dbounds}, \eqref{dbounds''}, H\"older's inequality, and Young's inequality, we have
	\begin{equation}\label{in1}
	\begin{split}
	\norm{\nabla w_{\delta}}_{L^{2}(\mathbf{M}\,\backslash\, F)}^2&\leq\frac{1}{8}\norm{K_{\delta-}}_{L^{3/2}(\mathcal{O}_{\delta})}\norm{w_{\delta}}_{L^{6}(\mathbf{M}\,\backslash\, F)}^2\\
	&+\frac{1}{8}\norm{K_{\delta-}}_{L^{6/5}(\mathcal{O}_{\delta})}\norm{w_{\delta}}_{L^{6}(\mathbf{M}\,\backslash\, F)}\\
	&+\frac{1}{8}\norm{\left(\operatorname{div}_{\bar{\mathbf{g}}_{\delta}}X_{\delta}\right)_+}_{L^{3/2}(\bar{\Omega}\,\backslash\, F)}\norm{w_{\delta}}_{L^{6}(\bar{\mathbf{M}}\,\backslash\, F)}^2\\
	&+\frac{1}{8}\norm{\left(\operatorname{div}_{\bar{\mathbf{g}}_{\delta}}X_{\delta}\right)_+}_{L^{6/5}(\bar{\Omega}\,\backslash\, F)}\norm{w_{\delta}}_{L^{6}(\bar{\mathbf{M}}\,\backslash\, F)}\\
	&\leq\frac{1}{8}C_1{\delta}^{1/3}|\partial\bar{\Omega}|^{2/3}\norm{w_{\delta}}_{L^{6}(\mathbf{M}\,\backslash\, F)}^2+\frac{C_{\delta}}{32}d_3^2\\
	&+\frac{C_{\delta}}{32}C_1^2{\delta}|\partial\bar{\Omega}|^{5/3}+\frac{d_2}{8}\norm{w_{\delta}}_{L^{6}(\mathbf{M}\,\backslash\, F)}^2\\
	&+\frac{1}{8C_{\delta}}\norm{w_{\delta}}_{L^{6}(\mathbf{M}\,\backslash\, F)}^2+\frac{1}{8C_{\delta}}\norm{w_{\delta}}_{L^{6}(\mathbf{M}\,\backslash\, F)}^2\\
	&\leq\frac{1}{8}\left(C_1\delta^{1/3}|\partial\bar{\Omega}|^{2/3}+d_2+2C_{\delta}^{-1}\right)\norm{w_{\delta}}_{L^{6}(\mathbf{M}\,\backslash\, F)}^2\\
	&+\frac{1}{32}C_{\delta}C_1^2\delta|\partial\bar{\Omega}|^{5/3}+\frac{1}{32}C_{\delta}d_3^2.
	\end{split}
	\end{equation}
	Recall the Sobolev inequality \cite[Lemma 3.1]{SchoenYauI}  
	\begin{equation}\label{in2}
	\norm{w_{\delta}}_{L^{6}(\mathbf{M}\,\backslash\, F)}^2\leq C_{\delta}\norm{\nabla w_{\delta}}_{L^{2}(\mathbf{M}\,\backslash\, F)}^2,
	\end{equation}where $C_{\delta}$ is the Sobolev constant and it is uniformly close to Sobolev constant of $\bar{\mathbf{g}}$. For sufficiently small $\delta>0$, by Definition \ref{DefX},  $d_2$ satisfies
	\begin{equation}\label{in3}
	C_1\delta^{1/3}|\partial\bar{\Omega}|^{2/3}+d_2\leq \frac{2}{3C_{\delta}}.
	\end{equation}Combining equations \eqref{in1}, \eqref{in2}, and \eqref{in3}, we get the $L^6$ bound for $w_{\delta}$, that is 
	\begin{equation}\label{L6bound}
	\norm{w_{\delta}}_{L^{6}(\mathbf{M}\,\backslash\, F)}^2\leq \frac{3}{64}C_{\delta}^2C_1^2\delta|\partial\bar{\Omega}|^{5/3}+\frac{3}{64}C_{\delta}^2d_3^2.
	\end{equation} Then, we define $u_{\delta}=w_{\delta}+1\geq 1$. Moreover, the scalar curvature of the conformal metric $\hat{{\mathbf{g}}}_{\delta}=u_{\delta}^4\bar{{\mathbf{g}}}_{\delta}$ is
	\begin{equation*}
	\begin{split}
	R(\hat{{\mathbf{g}}}_{\delta})&=u_{\delta}^{-5}\left(R(\bar{{\mathbf{g}}}_{\delta})u_{\delta}-\frac{1}{8}\Delta_{\bar{\mathbf{g}}_{\delta}}u_{\delta}\right)\\
	&=u_{\delta}^{-5}\left(K_{\delta}u_{\delta}+2|\mathbf{X}_{\delta}|_{\bar{\mathbf{g}}_{\delta}}^2u_{\delta}
	-2\operatorname{div}_{\bar{\mathbf{g}}_{\delta}}\mathbf{X}_{\delta}u_{\delta}+|\bar{\mathbf{k}}_{\delta}|^2_{\bar{\mathbf{g}}_{\delta}}u_{\delta}-\frac{1}{8}\Delta_{\bar{\mathbf{g}}_{\delta}}u_{\delta}\right)\\
	&=u_{\delta}^{-5}\left(K_{\delta +}u_{\delta}+2|\mathbf{X}_{\delta}|_{\bar{\mathbf{g}}_{\delta}}^2u_{\delta}
	+2\left(\operatorname{div}_{\bar{\mathbf{g}}_{\delta}}\mathbf{X}_{\delta}\right)_-u_{\delta}+|\bar{\mathbf{k}}_{\delta}|^2_{\bar{\mathbf{g}}_{\delta}}u_{\delta}\right)\geq 0.
	\end{split}
	\end{equation*}Next we compute the ADM mass of the conformal metric 
	\begin{equation}\label{mass1}
	m_{ADM}(\hat{\mathbf{g}}_{\delta})=m_{ADM}(\bar{\mathbf{g}}_{\delta})+2A_{\delta}=m_{ADM}(g_+)+2A_{\delta},
	\end{equation}where $1\leq u_{\delta}=1+\frac{A_{\delta}}{|x|}+O(|x|^{-2})$ and $A$ is a positive constant 
	\begin{equation}
	\begin{split}
32\pi A_{\delta}&=\int_{\bar{\mathbf{M}}\,\backslash\, \bar{F}}\kappa (w_{\delta}+1)dV_{\bar{\mathbf{g}}_{\delta}}\\
	&=\int_{\mathcal{O}_{\delta}}K_{\delta-}w_{\delta}dV_{\bar{\mathbf{g}}_{\delta}}+\int_{\mathcal{O}_{\delta}}K_{\delta-}dV_{\bar{\mathbf{g}}_{\delta}}\\
	&+2\int_{\bar{\Omega}\,\backslash\, \bar{F}}\left(\operatorname{div}_{\bar{\mathbf{g}}_{\delta}}\mathbf{X}_{\delta}\right)_+w_{\delta}dV_{\bar{\mathbf{g}}_{\delta}}+2\int_{\bar{\Omega}\,\backslash\, \bar{F}}\left(\operatorname{div}_{\bar{\mathbf{g}}_{\delta}}\mathbf{X}_{\delta}\right)_+dV_{\bar{\mathbf{g}}_{\delta}}.
	\end{split}
	\end{equation}Then using \eqref{Kbound}, \eqref{dbounds}, \eqref{dbounds''}, \eqref{L6bound}, and H\"older's inequality, we have 
	\begin{equation}\label{A2}
	\begin{split}
32\pi A_{\delta}	&\leq \norm{K_{\delta-}}_{L^{6/5}(\mathcal{O}_{\delta})}\norm{w_{\delta}}_{L^{6}(\bar{\mathbf{M}}\,\backslash\, F)}+C_1\delta^{2/3}|\partial\bar{\Omega}|\\
	&+2\norm{\left(\operatorname{div}_{\bar{\mathbf{g}}_{\delta}}\mathbf{X}_{\delta}\right)_+}_{L^{6/5}(\mathcal{O}_{\delta})}\norm{w_{\delta}}_{L^{6}(\bar{\mathbf{M}}\,\backslash\, F)}+2d_1\\
	&\leq \left(C_1\delta^{1/2}|\partial\bar{\Omega}|^{5/6}+2d_3\right)\norm{w_{\delta}}_{L^{6}(\bar{\mathbf{M}}\,\backslash\, F)}\\
	&+C_1\delta^{2/3}|\partial\bar{\Omega}|+2d_1\\
	&\leq \frac{1}{8}\left(C_1\delta^{1/2}|\partial\bar{\Omega}|^{5/6}+2d_3\right)\left(3C_{\delta}^2C_1^2\delta|\partial\bar{\Omega}|^{5/3}+3C_{\delta}^2d_3^2\right)^{1/2}\\
	&+C_1\delta^{2/3}|\partial\bar{\Omega}|+2d_1\,.
	\end{split}
	\end{equation}For sufficiently small $\delta>0$, by Definition \ref{DefX}, $d_1$ and $d_3$ satisfy 
\begin{equation}\label{A1}
	\begin{split}
	&\frac{1}{8}\left(C_1\delta^{1/2}|\partial\bar{\Omega}|^{5/6}+2d_3\right)\left(3C_{\delta}^2C_1^2\delta|\partial\bar{\Omega}|^{5/3}+3C_{\delta}^2d_3^2\right)^{1/2}\\
	&+C_1\delta^{2/3}|\partial\bar{\Omega}|+2d_1\leq  2\int_{\partial\bar{\Omega}}X(\bar{\nu})dA_{\hat{\sigma}}\,.
	\end{split}
\end{equation}
	Combining this with \eqref{A2} and \eqref{mass1}, we get \eqref{mass2}, which completes the proof.
\end{proof}

We now have an asymptotically flat Riemannian manifold $(\bar{\mathbf{M}}\,\backslash\, \bar{F},\hat{\mathbf{g}}_{\delta})$ for $\bar{F}\subsetneq\bar{\Omega}$ and by running the inverse mean curvature flow in $\bar{\mathbf{M}}\,\backslash\, \bar{F}$, we can prove Theorem \ref{thm1.1}.\\\\

\begin{proof}[Proof of Theorem \ref{thm1.1}] Assume $\Omega$ is admissible. Let $\bar{E}\subsetneq \bar{\Omega}$ be a minimizing hull and without loss of generality we assume $m_{H}(\partial\bar{E})>0$. Let $\theta>0$ be given. We can find a connected set $\bar{F}\supset \bar{E}$ with
	smooth boundary $\partial\bar{F}$ such that $\bar{F}\subsetneq\bar{\Omega}$, $ m_{H}(\partial \bar{F})>0$ and 
	\begin{equation}\label{FandE}
	|\partial \bar{F}|_{\bar{g}}-\theta\leq |\partial\bar{E}|_{\bar{g}}\leq |\partial \bar{F}|_{\bar{g}}+\theta,\quad m_{H}(\partial\bar{E})\leq m_{H}(\partial \bar{F})+\theta.
	\end{equation}Then we attach the Jang graph and obtain a complete Riemannian manifold $$(\bar{\mathbf{M}},\bar{\mathbf{g}},\bar{\mathbf{k}},\mathbf{X})=(\overline{M}\cup M_+,\overline{g}\cup g_+,\bar{k}\cup k_+,X\cup X_+),$$ with ADM mass $m_{ADM}(g_+)$. We follow Lemma \ref{lem1} and Proposition \ref{prop2}, to obtain smooth Riemannian $(\bar{\mathbf{M}}\,\backslash\,\bar{F},\hat{\mathbf{g}}_{\delta})$ with non-negative scalar curvature. Let $\bar{F}'$ be strictly minimizing hull containing $\bar{F}$ in $(\bar{\mathbf{M}},\hat{\mathbf{g}}_{\delta}\cup \bar{g}_+)$. \\ \\ \indent 
	As in Theorem 3.1 of \cite{ShiTamHorizons}, $\bar{F}$ and $\bar{F}'$ are connected. By simple connectedness of the domain and the Seifert van Kampen theorem, $\partial\bar{\Omega}$ is homeomorphic to $S^2$. We now show that $\partial \bar{F}'$ is connected. \\ \indent 
	By the definition of $\bar{F}'$, we know that $\partial \bar{F}'$ seperates $\bar{\Omega}\cup \bar{M}$. Suppose, for contradiction, that $\partial F'$ has at least two components, $A$ and $B$. In that case, we may consider a closed curve $\gamma$ intersecting $A$ and $B$. By simple connectedness, there exists a map $f:D^2\to \bar{\Omega}\cup \bar{M}$ such that $f(\partial D^2)$ is a homeomorphism onto $\gamma$. Generically, $f^{-1}(A)$ is a compact properly embedded 1-manifold in $D^2$. But this cannot be since $\gamma$ only intersects $A$ once. Thus, $\partial \bar{F}'$ has a single component.\\ \\ \indent
	Returning to our extension and the flow, since the scalar curvature of $\hat{\mathbf{g}}_{\delta}$ has been made non-negative, the weak inverse mean curvature flow emanating from $\partial \bar{F}'$ produces
	\begin{equation}
	\begin{split}
	m_{{ADM}}(\hat{\mathbf{g}}_{\delta})&\geq \sqrt{\frac{|\partial \bar{F}'|_{\hat{\mathbf{g}}_{\delta}}}{16\pi}}\left(1-\frac{1}{16\pi}\int_{\partial \bar{F}'}\hat{H}_{\delta}dA_{\hat{\mathbf{g}}_{\delta}}\right)\\
	&\geq \sqrt{\frac{|\partial \bar{F}'|_{\hat{\mathbf{g}}_{\delta}}}{|\partial \bar{F}|_{\bar{g}}}}\sqrt{\frac{|\partial \bar{F}|_g}{16\pi}}\left(1-\frac{1}{16\pi}\int_{\partial \bar{F}}\bar{H}dA_{\bar{\mathbf{g}}}\right)\\
	&= \sqrt{\frac{|\partial \bar{F}'|_{\hat{\mathbf{g}}_{\delta}}}{|\partial \bar{F}|_{\bar{g}}}}m_{H}(\partial F)\\
	&\geq \sqrt{\frac{|\partial \bar{F}'|_{\hat{\mathbf{g}}_{\delta}}}{|\partial \bar{F}|_{\bar{g}}}}\left(m_{H}(\partial \bar{E})-\theta\right)\\
	&\geq \sqrt{\frac{|\partial \bar{F}'|_{\bar{\mathbf{g}}}}{|\partial \bar{F}|_{\bar{g}}}}m_{H}(\partial \bar{E})-\theta\sqrt{\frac{|\partial \bar{F}'|_{\hat{\mathbf{g}}_{\delta}}}{|\partial \bar{F}|_{\bar{g}}}}\\
	&\geq \sqrt{\frac{|\partial \bar{E}|_{\bar{g}}}{|\partial \bar{F}|_{\bar{g}}}}m_{H}(\partial \bar{E})-\theta C\\
	&\geq \sqrt{\frac{|\partial \bar{E}|_{\bar{g}}}{|\partial E|_{\bar{g}}+\theta}}m_{H}(\partial \bar{E})-\theta C.
	\end{split}
	\end{equation}The first inequality follows from the Geroch monotonicity of inverse mean curvature flow. The second inequality used the fact that mean curvature of $\partial \bar{F}'$ in $\hat{\mathbf{g}}_{\delta}$ is zero on $\partial \bar{F}'\,\backslash\,\partial F$ and the mean curvature of $\partial \bar{F}'$ and $\partial \bar{F}$  are the same $\hat{H}_{\delta}=\bar{H}u_{\delta}^{-2}$ a.e. on $\partial \bar{F}'\cap \partial \bar{F}$. The third inequality follows from equation \eqref{FandE} and the fourth inequality follows from $u_{\delta}\geq 1$. To get the fifth inequality we combine the following two facts. First, following the proof of Theorem 3.1 of \cite{ShiTamHorizons}, we obtain $|\partial \bar{F}'|_{\bar{\mathbf{g}}}\geq |\partial \bar{E}|_{\bar{g}}$. Second, if there is no blow-up in the Jang graph, $u_{\delta}$ is bounded on compact manifold $\bar{\Omega}$ and we get $\sqrt{\tfrac{|\partial \bar{F}'|_{\hat{\mathbf{g}}_{\delta}}}{|\partial \bar{F}|_{\bar{g}}}}\leq C$ for some constant $C<\infty$. On the other hand, if there is a blow-up in the Jang graph and we have asymptotically cylindrical end \cite{SchoenYauI,AnderssonMetzger}, by analyzing eigenvalues of the operator $\Delta_{\bar{\mathbf{g}}_{\delta}}+\frac{1}{8}\kappa$, we observe that $u_{\delta}(t,x)\sim a_0+\sum_{i=1}^{\infty}a_ie^{-\sqrt{\lambda_i}t}\xi_i(x)$, where $a_0> 1$ and $a_i$ are constants and $\lambda_i$ and $\xi_i(x)$ are eigenvalue and eigenfunction of the Laplacian operator at the cylindrical end, respectively. Therefore, $u_{\delta}$ is bounded on $\bar{\Omega}$ and we have $\sqrt{\tfrac{|\partial \bar{F}'|_{\hat{\mathbf{g}}_{\delta}}}{|\partial \bar{F}|_{\bar{g}}}}\leq C$. Clearly, the last inequality follows from equation \eqref{FandE}. Hence, as $\theta\to 0$ we obtain 
	\begin{equation}
	m_{ADM}(\hat{\mathbf{g}}_{\delta}) \geq m_{H}(\partial \bar{E}).
	\end{equation}Combining this with Proposition \ref{prop2}, we have 
	\begin{equation}\label{m1}
	m_{ADM}(g_+)+\frac{1}{8\pi}\int_{\partial\bar{\Omega}}X(\bar{\nu})dA_{\hat{\sigma}}\geq m_{H}(\partial \bar{E}).
	\end{equation}By Wang and Yau \cite{WangYauPMT1} and admissible condition $X(\bar{\nu})>0$, we have
	\begin{equation}\label{m2}
	\begin{split}
	m_{WY}(\partial \Omega)&\geq \frac{1}{8\pi}\int_{\partial\bar{\Omega}}\left(\hat{H}_0-\left(\bar{H}-X(\nu)\right)\right)dA_{\hat{\sigma}}\\
	&\geq \frac{1}{8\pi}\int_{\partial\bar{\Omega}}\left(\hat{H}_0-\bar{H}\right)dA_{\hat{\sigma}}+\frac{1}{8\pi}\int_{\partial\bar{\Omega}}X(\bar{\nu})dA_{\hat{\sigma}}\\
	&\geq m_{ADM}(g_+)+\frac{1}{8\pi}\int_{\partial\bar{\Omega}}X(\bar{\nu})dA_{\hat{\sigma}}\,.
	\end{split}
	\end{equation}Combining \eqref{m1} and \eqref{m2}, the desired inequality for the Wang-Yau quasi-local mass in Theorem \ref{thm1.1} is achieved. A similar argument leads to the inequality for the Liu-Yau quasi-local mass. 
\end{proof}
\begin{proof}[Proof of Proposition \ref{cor1}] If the projection of MOTS $S$ is outward minimizing surface $\bar{S}$ in the blow up Jang graph, then applying Theorem \ref{thm1.1}, we have 
	\begin{equation}
	m_{WY}(\partial\Omega)\geq m_{H}(\bar{S})=\sqrt{\frac{|\bar{S}|}{16\pi}}\,.
	\end{equation}But since $|\bar{S}|\geq |S|$, we get the result. Similarly, the inequality follows for outward minimizing minimal surfaces in a Jang graph.

\end{proof}

\begin{proof}[Proof of Theorem \ref{thm1.4}]
The argument is by contradiction and roughly follows from combing the radius definition of Schoen and Yau \cite{SchoenYauIII} and Theorem 3.2 of \cite{ShiTamHorizons}, with the main difference being that $m^*(\Omega;\bar{\Omega})$ does not detect minimal surfaces, unlike $m_{ST}(\Omega)$ in Riemannian setting. \\ \indent 
Suppose that there is no MOTS in $\Omega$. Since $\partial\Omega$ is untrapped, then by Proposition 2 of \cite{SchoenYauIII} there exist a unique smooth bounded Jang graph $\bar{\Omega}$ solution to the Dirichlet problem \eqref{Jang}. Now consider some $\bar{\Omega}_1 \subsetneq  \bar{\Omega}_2\subset \bar{\Omega}$ with smooth boundaries, and let $\bar{E}$ be a connected minimizing hull of $\bar{\Omega}_2$ with $C^2$ boundary such that $\bar{E}\subset \bar{\Omega}_1$. Wanting to prove that $m_{LY}(\partial \Omega)\geq m^*(\Omega; \bar{\Omega})$, it suffices to show that $m_{LY}(\partial \Omega)\geq \alpha^*_{\bar{\Omega}_1,\bar{\Omega}_2}m_H(\partial \bar{E})$.\\ \indent 
Let $\bar{E}'$ be the strictly minimizing hull of $\bar{E}$ in $\bar{\Omega}$, which we know exists since $\partial\bar{\Omega}$ is mean convex. By definition we have that $\partial \bar{E} \cap \partial \bar{V}\neq \emptyset$ for some connected strictly minimizing hull $\bar{V}$ of $\bar{\Omega}$ and $\partial\bar{E}'$ is connected. By the definition of strictly minimizing hulls and the Regularity Theorem 1.3 of \cite{HuiskenIlmanen}, we know that $\partial \bar{E}'$ is $C^{1,1}$, and that $H_{\partial \bar{E}'\backslash \partial \bar{E}}=0$ and the mean curvature $H$ of $\partial \bar{E}$ and $\partial \bar{E}'$ are equal when $\partial \bar{E}$ and $\partial \bar{E}'$ intersect. Using Theorem 1.2, it follows that 
\begin{equation}\label{in12}
\begin{split}
m_{LY}(\partial \Omega)&\geq m_H(\partial \bar{E}')=\sqrt{\frac{|\partial \bar{E}'|}{16\pi}}\left( 1-\int_{\partial \bar{E}'} H^2dA\right)\\
&=\sqrt{\frac{|\partial \bar{E}'|}{|\partial \bar{E}|}}m_H(\partial \bar{E}) \,.
\end{split}
\end{equation}
In general, it could be that $\partial \bar{E}$ and $\partial \bar{E}'$ do not intersect. This only occurs when there is an outward minimizing minimal surface outside of $\bar{E}$. When such a jump occurs, then is no general relation between the mean curvature or indeed area of $\partial \bar{E}$ and $\partial \bar{E}'$ (apart from the fact that $H=0$ on $\partial \bar{E}'$). This scenario is forbidden by the fact that $\partial \bar{E}\cap \partial \bar{V}\neq \emptyset$ for some connected strictly minimizing hull $\bar{V}$ of $\bar \Omega$. \\ \indent
Therefore there are only two cases, either $\partial \bar{E}'\subset \bar{\Omega}_2$ or $\bar{\Omega}\backslash \bar{\Omega}_2 \cap \partial \bar{E}'\neq \emptyset$. The former case occurs when $\bar{E}'$ is actually also a strictly minimizing hull with respect to $\bar{\Omega}_2$, and in that case since $\bar{E}$ is also minimizing with respect to $\bar{\Omega}_2$, it follows by definition that $|\partial \bar{E}'|\geq |\partial \bar{E}|$, which in turn yields $m_{LY}(\partial \Omega)\geq m_H(\partial \bar{E})\geq m^*(\Omega;\bar{\Omega})$, where we have used that ${\alpha^*}_{\bar{\Omega}_1,\bar{\Omega}_2}\leq 1$. \\ \indent 
If $\partial E'\not\subset \bar{\Omega}_2$, then we are in the setting of the Theorem 3.2 of \cite{ShiTamHorizons}, which due to Meeks and Yau \cite{MeeksYau} we can estimate minimal surface area part of $\partial\bar{E}'$. In this case, we have 
\begin{equation}
\partial\bar{E}'\cap\bar{\Omega}_1\supset \partial\bar{E}'\cap\partial\bar{E}\neq\emptyset.
\end{equation}Therefore, the distance $d$ of $\bar{\Omega}_1$ and $\partial\bar{\Omega}_2$ is large enough such that there exist a $x$ in minimal part of $\partial\bar{E}'$ with $d(x,\partial\bar{\Omega}_2)=\frac{d}{2}$. By Lemma 2.4 of \cite{MeeksYau}, for any ball $B_x(r)$ centred at $x$ with radius $r=\min\{\iota,\frac{d}{2}\}$, where $\iota$ is the infinmum of the injectivity of points in $\{x;\; d(x,\partial\bar{\Omega}_2)>\frac{d}{4}\}$, we have
\begin{equation}
\begin{split}
|\partial \bar{E}'|&\geq |\partial \bar{E}'\cap B_x(r)|\geq CK^{-2}\int_0^r\tau^{-1}\sin(K\tau)^2d\tau\\
	&= {\alpha^*}^2_{\bar{\Omega}_1,\bar{\Omega}_2}|\partial\bar{\Omega}_1| {\beta_{\Omega,\bar{\Omega}}}\\
		&\geq  {\alpha^*}^2_{\bar{\Omega}_1,\bar{\Omega}_2}|\partial\bar{E}|.
\end{split}
\end{equation}The second inequality is by Meeks and Yau \cite{MeeksYau} and the last inequality follows from ${\beta_{\Omega,\bar{\Omega}}}\geq 1$ and the fact that $\bar{E}$ is a minimizing hull with respect to $\bar{\Omega}_2$. Hence combining this with inequality \eqref{in12}, we obtain $m_{LY}(\partial \Omega)\geq m^*(\Omega;\bar{\Omega})$ which is a contradiction. The inequality involving $m_{WY}(\partial\Omega)$ follows similarly.

Now to obtain the inequality involving $\text{diam}(\partial\Omega)$ and $m_{LY}(\partial\Omega)$, we use a standard Minkowski formula in $\mathbb{R}^3$. Let $S$ be the smallest sphere that circumscribes and envelops $\Omega$ in $\mathbb{R}^3$. Let $\mathbf{x_0}$ be its centre and $R$ its radius. Using the formula for $m_{LY}(\partial \Omega)$ and equation (ii), Lemma 6.2.9, pg. 136 of Klingenberg \cite{Klingenberg}, we have by Gauss-Bonnet theorem
\begin{equation}\label{diameter}
\begin{split}
m^*(\Omega;\bar{\Omega})&\leq m_{LY}(\partial \Omega)\\
&<\frac{1}{8\pi} \int_{\partial \Omega} H_0\\
&\leq \frac{R}{8\pi} \int_{\partial \Omega} K_{\partial\Omega}\\
& = \frac{1}{2} R \leq \frac{1}{4}\text{diam}(\partial \Omega),
\end{split}
\end{equation}
where we use the positivity of Gauss curvature $K_{\partial\Omega}$ in the third inequality. This is a contradiction and complete the proof.
\end{proof}
Before proving Proposition 1.5, we recall the following Lemma from \cite[Lemma 3.6]{ShiTamHorizons}.
\begin{lem}\label{Lem3}
	Let $(\Omega,g)$ be a compact Riemannian 3-manifold with smooth mean convex boundary $\partial\Omega$ such that Gauss curvature of $\partial\Omega$ is positive. Suppose $E\subset\subset\Omega$ such that $\partial E$ is a isoperimetric surface. Then either $\Omega$ contains an outward minimizing minimal surface or $E$ is a minimizing hull.
\end{lem}
The proof of Propositions \ref{prop1.5} and \ref{prop1.6} are as follows.
\begin{proof}[Proof of Proposition \ref{prop1.5}]By Lemma \ref{Lem3} and Theorem \ref{thm1.1}, a isoperimetric surface $\bar{V}$ in $\bar{\Omega}$ satisfying $m_{H}(\bar{V})\geq m_{LY}(\partial\Omega)$ is not a minimizing hull in $\bar{\Omega}$, therefore, there is an outward minimizing minimal surface $\bar{S}$ in $\bar{\Omega}$. 
\end{proof}

\begin{proof}[Proof of Proposition \ref{prop1.6}]
Combining Theorem 1.2 and the argument of Theorem 3.2 of \cite{ShiTamHorizons}, if there are no minimal surfaces in $\bar{\Omega}$, then $m_{LY}(\partial \Omega)\geq m_{ST}(\bar{\Omega})$. Flipping the inequality yields the desired result and the diameter bound is similar to \eqref{diameter}. The argument for $m_{WY}(\partial \Omega)$ is identical if there exists an admissible function $\tau$. 
 \end{proof}

Finally, the proof of the Proposition \ref{prop1.7} is similar to \cite[Theorem 2.3]{Corvino}. 
\begin{proof}[Proof of Proposition \ref{prop1.7}]
		 Assume $\bar{\Omega}$ is not diffeomorphic to a ball in $\mathbb{R}^3$. By Meeks-Simon-Yau \cite[Theorem 1]{MeeksSimonYau}, there exist an outward minimizing minimal surface $\bar{S}=\mathbb{S}^2$ in $\bar{\Omega}$. For a given point $p\in \bar{S}$ let $\{e_1,e_2\}$ be a basis of $T_p \bar{S}$ in which the second fundamental form $\Pi$ is diagonal and denote the principal curvatures at $p$ by $\kappa_i$. The Gauss equation implies
\begin{equation}
K_{\bar{S}}={Rm}_{\bar{S}}(e_1,e_2,e_1,e_2)={Rm}_{\bar{\Omega}}(e_1,e_2,e_1,e_2)+\kappa_1\kappa_2.
\end{equation}where $K_{\bar{S}}$ is the Guass curvature of $\bar{S}$. Since $\bar{S}$ is minimal, $\kappa_1\kappa_2\leq0$ at $p$ and therefore we have
	\begin{equation}
	K_{\bar{S}}\leq C^2.
	\end{equation}By the Gauss--Bonnet Theorem and assumption $m_{LY}(\partial\Omega)<\frac{1}{2C}$, we have
	\begin{equation}
	4\pi=\int_{\bar{S}} K_{\bar{S}} \leq |\bar{S}|C^2<\frac{|\bar{S}|}{4m_{LY}(\partial\Omega)^2}.
	\end{equation}Combining this and Penrose inequality in Proposition \ref{cor1} for outward minimizing minimal surface $\bar{S}$ in $\bar{\Omega}$, we have the following contradiction
	\begin{equation}
		m_{LY}(\partial\Omega) < \left(\frac{|\bar{S}|}{16\pi}\right)^{1/2}\leq m_{LY}(\partial\Omega).
	\end{equation}Therefore, $\bar{\Omega}$ has no minimal surfaces. Furthermore, by Meeks-Simon-Yau \cite{MeeksSimonYau}, $\bar{\Omega}$ is diffeomorphic to a ball in $\mathbb{R}^3$. This means there is no blow-up in any Jang graphs of $\Omega$ and therefore, there is no MOTS in $\Omega$. This also shows that the solution of Jang's equation must be unique.
\end{proof}

\end{document}